\documentclass[11pt,a4paper]{amsart}
\usepackage{amsmath,amsthm,amssymb,amscd,epsfig,color}
\usepackage{amsmath ,amsthm,epsfig,amssymb,graphicx}
\usepackage{verbatim}
\usepackage{amssymb}

\begin{document}

\title{$G-$convergence, Dirichlet to Neumann maps and Invisibility}

\thanks{Supported by the ERC  307179, and the MINECO grants  MTM2011-28198 and SEV-2011-0087 (Spain)}
\author{Daniel Faraco}
\address{Departamento de Matem\'aticas - Universidad Aut\'onoma de Madrid
 and Instituto de Ciencias Matem\'aticas
CSIC-UAM-UC3M-UCM, 28049 Madrid, Spain} \email{daniel.faraco@uam.es}
\author{Yaroslav Kurylev}
\address{Dept. of Mathematics, University College London,
Gower Street, London WC1E 6BT, UK}\email{y.kurylev@ucl.ac.uk}

\author{Alberto Ruiz}
\address{Departamento de Matem\'aticas - Universidad Aut\'onoma de Madrid
 and Instituto de Ciencias Matem\'aticas
CSIC-UAM-UC3M-UCM, 28049 Madrid, Spain} \email{alberto.ruiz@uam.es}

\date{}

\allowdisplaybreaks \sloppy \theoremstyle{plain}
\newtheorem{Theorem}{Theorem}[section]
\newtheorem{Lemma}[Theorem]{Lemma}
\newtheorem{Remark}[Theorem]{Remark}
\newtheorem{Cor}[Theorem]{Corollary}
\newtheorem{question}[Theorem]{Question}
\newtheorem{example}[Theorem]{Example}
\newtheorem{Prop}[Theorem]{Proposition}
\theoremstyle{definition}

\newtheorem{Def}[Theorem]{Definition}
\newtheorem{Rem}[Theorem]{Remark}
\newcommand{\commentm}[1]{\marginpar{\footnotesize #1}}
\newtheorem{Prob}{Problem}
\numberwithin{equation}{section}
\def\halmos{{\ \vbox{\hrule\hbox{\vrule height1.3ex\hskip0.8ex\vrule}\hrule}}\par \medskip}
\def\Xint#1{\mathchoice
{\XXint\displaystyle\textstyle{#1}}%
{\XXint\textstyle\scriptstyle{#1}}%
{\XXint\scriptstyle\scriptscriptstyle{#1}}%
{\XXint\scriptscriptstyle\scriptscriptstyle{#1}}%
\!\int}
\def\XXint#1#2#3{{\setbox0=\hbox{$#1{#2#3}{\int}$}
\vcenter{\hbox{$#2#3$}}\kern-.5\wd0}}
\def\ddashint{\Xint=}
\def\dashint{\Xint-}
\def \p{\partial}
\def \e{\epsilon}
\def \beq{\begin{equation}}
\def \eeq{end{equation}}
\def \t{\tilde}
\def \r{\tilde{R}}
\def \l{\lambda}
\def \D{\mathbb{D}}
\def \div{\operatorname{div}}
\def \Re{\operatorname{Re}}
\def \Im{\operatorname{Im}}
\def \supp{\operatorname{supp}}
\def \diam{\operatorname{diam}}
\def \C{\mathbb{C}}
\def \N{\mathbb N}
\def\Lip{\operatorname{Lip}}
\def\R{\mathbb{R}}
\def\2L{\Lambda_{\tilde{\gamma}}}
\def\1L{\Lambda_{\gamma}}
\def \c{\overline}
\def \d{\partial_z}
\def \dc{\partial_{\overline z}}
\def \cd{\overline{\partial_z}}
\def \dk{\partial_{\overline k}}
\def \kd{\overline{\partial_k}}
\def \dx{\partial_x}
\def \H{H^{1/2}}
\maketitle

\begin{abstract}
We establish optimal conditions under which the G-convergence  of 
linear elliptic operators implies the convergence
of the corresponding Dirichlet to Neumann maps. As an application we show that the 
  approximate cloaking  isotropic materials from  \cite{GKLU} are independent of the source.

\end{abstract}

\section{Introduction}

We start with the definition of the   Dirichlet to Neumann map
(Voltage to current) map.  Given  an elliptic matrix $\sigma\in L^\infty(\Omega)$, for a given  boundary data
$\varphi\in\H(\partial \Omega)$, there is a unique
 solution $u\in H^1(\Omega)$ to the Dirichlet problem;
\begin{equation}\label{cond}
\begin{cases}
\nabla\cdot(\sigma\nabla u)=0 \quad \hbox{in}\,\, \Omega\\
u\big|_{\partial \Omega}=\varphi.
\end{cases}
\end{equation}

When the boundary is sufficiently smooth, the
 measurements on the boundary consist of the classical
 Dirichlet--to--Neumann map
 \begin{equation}\label{trad}\Lambda_\sigma(\varphi)=\langle \sigma\nabla u, \nu \rangle \big|_{\partial \Omega},\end{equation}
 where $\nu$ denotes the exterior unit normal to the boundary. In
 this way
 $\Lambda_\sigma:H^{1/2}(\partial \Omega) \to
 H^{{-1/2}}(\partial \Omega)$. It follows by integration by
 parts that $\Lambda_\sigma$ can also be described in the weak
 form as 
\begin{equation}\label{weak formulation}
\langle \Lambda_{\sigma}(\varphi),\psi \rangle=\int_{\Omega}\langle\sigma \nabla u,
\nabla \tilde{\psi}\rangle,
\end{equation}
where $\psi \in H^{1/2}(\partial \Omega)$ and
$\tilde{\psi} \in H^1( \Omega)$ is an extension of $\psi$ into $\Omega$. 
In case $\partial \Omega$
lacks of a proper normal, the weak formulation is still valid.

The Calder\'on inverse problem consists of the stable determination of $\sigma$ from $\Lambda_\sigma$, see \cite{SylvesterUhlmann87, HabermanTatarupp, Nachman96, AstalaPaivarinta06}  for 
the uniqueness 
in the isotropic case, 
\cite{Alessandrini88, Alessandrini90, BBR01,BFR07,CFR11, FaracoRogers12,CGR13}
 for stability and \cite{Nachman88,Nachman96} for the reconstruction.
 Much less is known in the anisotropic case except in dimension d=2 \cite{ALP}. Notice that when  the Dirichlet to Neumann map is known for all energies, uniqueness and  stability are  studied also for the anisotropic case, see  e.g. \cite{KK, AK2LT}.  
 
  The results from 
 \cite{Alessandrini88, BBR01, BFR07,CFR11, FaracoRogers12,CGR13} require  some uniform control  of the oscillations of $\sigma$ (conditional stability). 
 Unfortunately,  wild oscillations of a sequence of conductivities  $\sigma_n$ creates 
 an instability of the Calder\'on problem. This is well
expressed in terms of the G-topology \cite{D, JKO}.  It is not hard to see that  if $\sigma_n$ G-converges to $\sigma$, the
corresponding  Dirichlet to Neumann maps  converge weakly. Namely,  
for each $\varphi,\psi \in H^{1/2}(\partial \Omega)$,  

\begin{equation}\label{weak_0}
\langle \Lambda_{\sigma_h}(\varphi),\psi \rangle \to \langle \Lambda_{\sigma}(\varphi),\psi \rangle.
\end{equation}

Now, if $\sigma_n$ G-converges to $\sigma$ but does not convergence pointwise, we deduce that 
the convergence (\ref{weak_0}) does 
not imply any sort of $L^p$ convergence. (Notice  $\sigma_n,\sigma$
could be choosen to be $C^\infty$!). 

However, the stability
estimates are normally stated in terms of the operator norm and
(\ref{weak_0})  by itself does not imply the convergence in the operator
norm $\|   \|_{\mathcal{L}(H^{ {1/2}}(\partial \Omega) \to
H^{-1/2}(\partial \Omega))}$.  In \cite{AC1}, it is proved that if, in addition to the $G-$convergence, we have that $\sigma_n=\sigma=I$ on 
 $\Omega_\delta=\{x \in \Omega: d(x,\partial \Omega) \le \delta\}$, with $\Omega$ 
 being the unit disc  and $\sigma=I$, 
 then in fact the $G$-convergence implies the convergence in the operator norm.  
On the other hand,  the stability at the boundary of the inverse problem implies that, in order to obtain operator norm convergence, 
some control on the behaviour of the conductivities at the boundary is needed. For example,
 it was proved, see   \cite{SU},  \cite{Brown01}, \cite{AG} and \cite{GZ}, that, for isotropic conductivities, if
 $\lim_{n \to \infty}\|\Lambda_{\rho_n} \to \Lambda_\rho\|_{H^{1/2} \to H^{-1/2}(\partial \Omega)}=0$
  then 
 \begin{equation} \label{necessary}
 \lim_{n \to \infty} \|\rho_n -\rho \|_{L^\infty(\partial \Omega)}=0   
 \end{equation} 
 
 Thus, the G-convergence by itself can not guarantee 
 the operator norm convergence. Let $\Omega \subset \mathbb{R}^n$ and
 define, for  $K\geq 1, \delta>0$,
\begin{equation}
\begin{aligned}
M_K(\Omega)=\{& \sigma\in L^\infty ( \Omega, \,{\mathcal M}^{n\times n}):\frac 1K|\xi|^2\leq \sigma \xi\cdot \xi\leq K|\xi|^2 \\ 
&\text {for almost  every $x\in \Omega$ and  $\xi\in \mathbb R^n$} \};\\
& \Omega_\delta= \{x\in \Omega: d(x,\partial \Omega)<\delta\}.
\end{aligned}
\end{equation}
 
The following theorem seems to be essentially sharp (see comments below).

  \begin{Theorem}\label{general_0} Let $\Omega \subset \mathbb{R}^d$ be a domain. 
Assume that 
\begin{equation}\label{sharp_0}
\lim_{\delta \to 0} \delta^{-1} \left( \limsup_{n \to \infty} \|\sigma_n-\sigma\|_{L^{\infty}(\Omega_\delta)} \right)
=0
\end{equation}
and that $\sigma_n \in M_K$ converges to $\sigma$ in the sense of the
$G-$convergence. Then 
\[\lim_{n \to \infty} \|\Lambda_{\sigma_n}-\Lambda_{\sigma} 
\|_{\H(\partial \Omega) \to H^{-1/2}(\partial \Omega)}=0.\]
\end{Theorem}
 
Let us emphasize that no regularity assumption is made on the domain or on the conductivities.  The condition (\ref{sharp_0})
can be read as a weak version of 
\begin{equation}
\lim_{n \to \infty} (\|\nabla_\nu (\sigma_n-\sigma) \|_{L^\infty(\partial \Omega)}+\| \sigma_n-\sigma \|_{L^\infty(\partial \Omega)})=0.
\end{equation}

 Note that the above conditions are   natural,  since the convergence of the D-N maps 
is known to imply   convergence 
of the 
conductivities  and their normal derivatives at the boundary  under mild regularity assumptions (\cite{SU,Alessandrini90,Brown01,AG, GZ}).  Moreover,  in Theorem
4.9 we provide an explicit example which shows that, just the convergence $\sigma_n$ to $\sigma$ in $L^\infty(\partial \Omega)$  together with 
their convergence in $L^p(\Omega)$, for any $p<\infty$,
are not sufficient for the norm-convergence of the DN maps.

 The proof of Theorem \ref{general_0} 
is very different in spirit
to that from \cite{AC1} and we believe it to be of an independent interest. The proof in  \cite{AC1} uses
 the decay properties of 
 the spherical harmonics away from the boundary.  Under 
 some regularity assumptions on $\sigma$, 
which  in turn imply certain properties  of the corresponding Poisson kernel,
a related strategy works (estimating decay properties of solutions with oscillating boundary data away from the boundary). 
To prove theorem \ref{general_0}  
we argue in a different manner.  Namely, we 
sudy the behaviour of  the solutions near the boundary. 
If, for example, $\sigma_n=\sigma$ on $\Omega \setminus \Omega'$, where $\Omega' \subset \subset \Omega$,
then
the difference of two solutions of the Dirichlet problem associated with   $\sigma_n$ and $\sigma$
 solves the same elliptic equation in $\Omega \setminus \Omega'$. It turns out that the resulting operators from the boundary into
$\Omega \setminus \Omega'$
are compact in a proper space. Our way to codify this is to factorize 
$\Lambda_{\sigma_n}-\Lambda_{\sigma}= T \circ A_n$, where $T$ is compact.  The arguments
behind this procedure are quite robust and allow to relax the condition 
$\sigma_n=\sigma$ on $\Omega \setminus \Omega'$ to (\ref{sharp_0}).

Next we turn to applications of our techniques to what is called an approximate cloaking.
 In the last decade it has been shown that the failure of uniqueness in 
 the Calder\'on problem is related to the modeling of invisible materials and 
 what is called acoustic and electromagnetic cloaking, see \cite{GLU, PSS, L, GKLU1}. It 
 is shown there that
the available conductivities yielding perfect cloaking are singular and anisotropic. 
Recently it has been shown that they can be approximated by elliptic isotropic materials in the G-convergence sense \cite{GKLU},
see also \cite{KV1, NV, LLRU} for different approachs. Leaving precise formulations of the involved operators and a general case to section 3, 
assume that $\Omega=B_3$, i.e. the ball of radius $3$ in $\R^3$ and $q \in L^\infty(B_1)$ is an arbitrary potential.
Consider the D-N maps associated with
the Dirichlet problems, with spectral parameter $\lambda$, for the free space, 
\begin{equation}\label{cond_1}
\begin{cases}
-\nabla\cdot\nabla u=\lambda u\\
u\big|_{\partial \Omega}=\varphi,
\end{cases}
\end{equation}
cf. (\ref{cond}), and for cloaked space,
\begin{equation}\label{cond_n}
\begin{cases}
-g_n^{-1/2}\nabla \cdot \sigma_n\nabla u+qu=\lambda u\\
u\big|_{\partial \Omega}=\varphi.
\end{cases}
\end{equation}
Here 
the weight factors $g_n$ and the isotropic conductivities $\sigma_n=\gamma_n I$, which are supported
in $\{x:\, 1 \leq |x|\leq 2 \}$, are chosen independent of $q$. Denoting by $\Lambda^\l_{out}$ and 
$\Lambda^\l_n$ the corresponding D-N operators, we have

 \begin{Theorem} \label{invisibility}
The exists a sequence $g_n,\, \gamma_n$ such that, for all except a countable number of $\l$,
$$
\|\Lambda_{n}^\l-\Lambda_{out}^\l \|_{H^{1/2}(\p \Omega) \to H^{-1/2}(\p \Omega)}
\to 0, \quad \hbox{as}\,\, n \to \infty.
$$
\end{Theorem}
This theorem means that, by a proper choice of $g_n,\, \gamma_n$, one can better 
and better hide from an external
observer any potential in $B_1$.
The novelty here is that we have an operator norm convergence instead of the strong 
convergence in
\cite{GKLU}.  We refer the reader to section 3 for the more general case as well as 
the explanation of the nature of
the exceptional points $\l$.

Let us note that, motivated by the applications to acoustic, quantum and electromagnetic 
cloaking, we extend our results 
of the type of Theorem \ref{general_0}, to the  operators  
   
\begin{equation} \label{8.1.4}
\L_n u=-\nabla_{A_n}\cdot \sigma_n \nabla_{A_n}u + q_n u, \quad u|_{\p \Omega}=0,
\end{equation} \label{operator-L}
where $\nabla_A= \nabla+i A$ with $A$ being a real one-form.
However, for the sake of brevity, we do so only for the case 
when $\sigma_n=\sigma, \, A_n=A,\, q_n=q$ in 
$\Omega \setminus \Omega'$.

Finally, we point out that, since the G-convergence rules out general stability results 
with respect to $L^p$ classes, one is tempted to conjecture that the convergence of the 
 D-N maps implies the G-convergence.  Recall that if $F$ is a diffeomorphism
 of $\Omega$, which is the 
 identity at the boundary, then $\Lambda_{F^*(\sigma)}=\Lambda_\sigma$. As   discussed for 
 example in \cite{AC2}, the isotropic conductivities are G-dense in the 
 set of anisotropic 
conductivities, so that
 the only hope is to recover from the D-N maps the G-limit up to a gauge transformation. 
In contrast to the previous results on the conditional stability, the 
compactness of the sets $M_K$ in the G-topology indeed provides a  stability result 
which is unconditional respect to regularity (we still require ellipticity).

\begin{Theorem} \label{stability}
Let $d=2, \sigma_n \in M_K$. Then   
\begin{equation} \label{29.1.07}
\lim_{n \to \infty} \Lambda_{\sigma_n}=\Lambda_{\sigma},
\end{equation}
weakly in  $H^{-1/2}(\partial \Omega)$ if and only if  there exists a sequence of quasiconformal maps 
 $F_n:\Omega \to \Omega,\, F_n|_{\p \Omega}= id|_{\p \Omega},$  such that,
 in the sense of the $G-$convergence,
\begin{equation} \label{3.17.4}
F_n^*(\sigma_n) \to \sigma. 
\end{equation}
\end{Theorem}

Let us emphasize that since there is no requirement at the boundary here we speak only of weak convergence of the D-N maps.

The paper is structured as follows. In section 2 we start by proving the convergence of the 
D-N maps for
the operators of form (\ref{8.1.4}),  
assuming that $\sigma_n=\sigma, \, A_n=A,\, q_n=q$ in  a neighborhood of the boundary,
see Theorem~\ref{2.1}. Note that
 this is the case which will be needed for applications to aprroximate cloaking considered
 in section 3. In section 4 we prove Theorem~\ref{general_0} and in section 5 we prove Theorem~\ref{stability}.

{\bf{Acknowledgments:}} We thank G.Alessandrini, R.Brown and J.Sylvester 
 for inspiring conversations on the problem. We also thank G.Alessandrini   for suggesting that a condition similar 
to (\ref{sharp_0})
might hold. The research started during visit of the three authors to the 
Isaac Newton Institute in Cambridge during the program 
"Inverse problems" in 2011, was continued during several visits of the second author to Madrid and  during the program 
"Inverse problems and applications" at the Mittag-Leffer Institute in Stockholm in 2013. 
The second author would also like to thank ACMAC, Heraklion which he
visited during the final stage of the preparation of the manuscript. We would like to thank for the fantastic 
research enviroment in all these occasions.

\begin{section}{Operators which coincide near the boundary}

Let   $\Omega \subset \mathbb{R}^n$ be any bounded domain.
We consider the  conductivity  equations with magnetic potential $A_n$ and electrical 
potential 
$q_n\in L^\infty(\Omega)$ and   the spectral parameter $\lambda\in \mathbb C$. Namely, for 
 $u \in H^1(\Omega, \mathbb{C})$ we define the Dirichlet problem for the corresponding 
differential operator $\L_n$: 
\begin{equation}
\L_n u=-\nabla_{A_n}\cdot \sigma_n \nabla_{A_n}u + q_n u, \quad u|_{\p \Omega}=0.
\end{equation} \label{operator-L}
Here 
\begin{equation} \label{1.1.03}
\sigma_n\in M_K,\, K>1, \quad \hbox{i.e.}\,\, \frac{1}{K} I \leq \sigma_n(x) \leq K I,\, x \in \Omega,
\end{equation}
and
\begin{equation} \label{2.1.03}
A_n \in L^\infty (\Omega, \R^d),\, \, \|A_n\|_\infty \leq  K, \quad
q_n \in L^\infty (\Omega, \R),\,\,\|q_n\|_\infty \leq K,
\end{equation}
 where for simplicity we assume all $K's$ to be the
same. The magnetic gradient  is given by
$\nabla_{A_n}u= \nabla  u +iA_n u.$
Note that conditions (\ref{1.1.03}), (\ref{2.1.03}) imply the existence of $\l(K)$, such that
$$
(-\infty, \l(K)) \cap \hbox{spec}(\L_n)= \emptyset.
$$
If  $\lambda \notin \hbox{spec}(L_n)$, then for a given boundary data
$\psi\in\H(\partial \Omega)$ there is a unique
 solution $u_n=u_n^\psi(\l) \in H^1(\Omega)$ to the Dirichlet problem;
\begin{equation}\label{condmagn}
\begin{cases}
\L_n^\lambda u_n:=(\L_n-\l) u_n=0\\
u_n\big|_{\partial \Omega}=\psi.
\end{cases}
\end{equation}

For the  regular domains we define  the Dirichlet to Neumann map, 
$\Lambda_n^\lambda: \H(\p \Omega) \to H^{-1/2}(\p \Omega) $ by
$$\Lambda_n^\lambda(\psi) = \nu\cdot \sigma_n \nabla_{A_n} u_n.$$
It follows by integration by
 parts that $\Lambda_n^\lambda$ can also be described in the weak
 form as
\begin{equation}\label{weak formulation}
\langle \Lambda_n^\lambda(\psi),\varphi \rangle=\int_{\Omega}(\sigma_n \nabla_{A_n} u_n\cdot
\overline{\nabla_{A_n} \tilde{\varphi}} + (q_n-\lambda) u_n\overline{\tilde\varphi}),
\end{equation}
where $\varphi \in \H(\partial \Omega)$ and
$\tilde{\varphi} \in H^1(\Omega)$ is an extension of ${\varphi} $ to $\Omega$. 
We will denote the solution of (\ref{condmagn}) by 
$u^\psi_n$ or even $u_n$  (we omit the  dependence on $\lambda$). In this  section  we  prove
that

\begin{Theorem} \label{2.1} Let $\L _n,\, \L$ be operators of form (\ref{operator-L})
which satisfy (\ref{1.1.03}), (\ref{2.1.03}). Assume that there is $\Omega'$
with ${\overline \Omega'} \subset \subset \Omega$ such that 
$\sigma_n=\sigma $, $A_n=A$ and $q_n=q$ on $\Omega\setminus \Omega'$.
Then, if the operators $\L _n$ $G$- converge to $\L$, we have
\begin{equation} \label{28.3.1}
\|\Lambda_n^\lambda - \Lambda_0^\lambda\|_{H^{1/2}\to H^{-1/2}} \to 0.
\end{equation} 
Here $\l \notin \hbox{spec}(\L)$ and, for any  $\Bbb K \subset \C$  being a compact  set such that 
$\Bbb K\cap \hbox{spec}(\L)= \emptyset$, 
 the convergence is uniform for  $\lambda\in \Bbb K$.
\end{Theorem}

The proof of the theorem is rather long and will consist of
several steps.  

Let us first note that
there are several equivalent definitions of the $G-$convergence of operators, see e.g. Th. 13.6  and Example 13.13, \cite{D} which essentially
amount to the convergence of the solutions. For our purpose we use

\begin{Def}
The operator $\L _n$  G-converges to $\L$ 
if, for any $f \in H^{-1}(\Omega)$ and $\l< \l(K)$, it holds that 
\begin{equation}\label{resolvent}
(\L_n-\l I)^{-1} f \to (\L-\l I)^{-1} f ,\quad \hbox{as}\,\, n \to \infty,
\end{equation} 
weakly in $H_0^1(\Omega)$.  
\end{Def}

The following two lemmata follow more or less straightforward from the definition of the 
G-convergence. To this end we first introduce the quadratic form, $\ell_n$,
associated with $\L_n$,
\begin{equation} \label{q-form} 
\ell_n [u]= \int_{\Omega}\left(\sigma_n \nabla_{A_n}u, \nabla_{A_n} u \right) +q_n|u|^2,\,\,
u \in H^1_0(\Omega).
\end{equation}

\begin{Lemma} \label{K}
Let now $\l \in \Bbb K$. Then, for any $f \in H^{-1}(\Omega)$, 
\begin{equation}\label{resolvent1}
(\L_n-\l I)^{-1} f \to (\L-\l I)^{-1} f ,\quad \hbox{as}\,\, n \to \infty,
\end{equation}
weakly in $H_0^1(\Omega)$ and uniformly in $\Bbb K$.
\end{Lemma}

\begin{proof}

Using the coercivity of 
$\ell_n$  the proof follows  the lines of \cite[Lemma 2.7]{GKLU}. 
Note that this fact does not require  the coincidence of $\sigma_n, \,A_n$
and $q_n$ with $\sigma, \,A$
and $q$ near $\p \Omega$ since it follows from the uniform ellipticity of 
forms $\ell_n$ together with the strong resolvent convergence of 
(\ref{resolvent}).
\end{proof}

\begin{Lemma}\label{Gconvergencia} 
Let $\L_n$    $G$-converge to $\L$
and $\l \notin \hbox{spec}(\L)$. Then,  for each $\psi \in {\mathcal H}^{1/2}(\partial\Omega)$,
\[  (\Lambda^\lambda_{n} -\Lambda^\l )(\psi) \rightharpoonup  0 \textrm{ in } {\mathcal H}^{-1/2}. \]
\end{Lemma}
\begin{proof} Denoting by $u_n,\,u$ the solutions to (\ref{condmagn}), we
have, by Theorem 22.9 \cite{D}, that
$$
\ell_n[u_n]-\l \|u_n\|^2 \to \ell[u]-\l \|u\|^2, 
\quad \hbox{as}\,\, n \to \infty.
$$
Polarising this equality, we arrive at 
$$
\langle \Lambda_n^\lambda(\psi),\varphi \rangle \to 
\langle \Lambda^\lambda(\psi),\varphi \rangle.
$$
\end{proof}

  We denote by $ \mathring{H}^{1}(\Omega\setminus\Omega')$ the functions in $H^{1}(\Omega\setminus\Omega')$ with trace $0$ on $\partial \Omega$.
A key fact in our arguments is  the following Caccioppoli  type inequality. 

\begin{Lemma}\label{Caccioppoli} Let $w\in \mathring{H}^{1}(\Omega\setminus\Omega')$ be  a weak solution  of  
\begin{equation}\label{3.1.03}
\L^\lambda w=f +\div F \textrm{ on }\Omega\setminus\Omega',
\end{equation}
  for $f\in L^2( \Omega)$ and  $F$ being 
  a vector field in $L^2(\Omega\setminus \Omega' )$. Then, for 
  any $\Omega''$, 
  ${\overline \Omega'} \Subset \Omega'',\, {\overline \Omega''} \Subset \Omega$, there exists  a $C=C(\Omega, \Omega',\Omega'', K, \lambda)$ 
  such that 
\begin{equation}\label{equationring}
\int_{\Omega\setminus \Omega''}|\nabla w|^2 
\leq C \left(\int_{\Omega\setminus \Omega'}|w|^2+
\int_{\Omega\setminus \Omega'}|F|^2 +
\int_{\Omega\setminus \Omega'}|f|^2 \right).
 \end{equation}
 
 Moreover if we choose $\Omega'=\Omega_{2\delta}$ and $ \Omega''=\Omega_\delta$ the estimate is  
 \begin{equation}\label{equationring2}
\int_{\Omega_\delta}|\nabla w|^2 
\leq C \left(\delta^{-2} \int_{\Omega_{2\delta}}|w|^2+
\int_{\Omega_{2\delta}}|F|^2 +
\int_{\Omega_{2\delta}}|f|^2 \right).
 \end{equation}

\end{Lemma} 
\begin{proof} 

 Choose  $\eta \in {C}^\infty(\mathbb{R}^n \setminus \Omega')$ such that $ \eta=1$ on 
 $\Omega\setminus \Omega''$ and $\eta=0$ near $\p \Omega'$. 
 Since  $w$ has zero trace on 
 $\partial  \Omega$,  $\eta^2 w \in H_0^{1}(\Omega \setminus \Omega')$. Thus it is a proper 
 test function for  the weak formulation of (\ref{3.1.03}). Thus, 
\[ \int_{\Omega \setminus \Omega'} \langle \sigma \nabla_A w, \nabla_A (\eta^2w)\rangle  =\int_{\Omega \setminus \Omega'} \langle F, \nabla(\eta^2 w)\rangle  -\int_{\Omega \setminus \Omega'} \eta^2(q-\lambda) |w|^2  +\int_{\Omega \setminus \Omega'}f  \overline {\eta^2w}\]
Hence 
$$ \big|\int_{\Omega\setminus \Omega'}\eta^2\langle \sigma \nabla_A w, \nabla_A w\rangle \big|$$
$$ \leq \underbrace{\big|2\int_{\Omega\setminus \Omega'}\eta \langle \sigma \nabla_A w, \nabla \eta\rangle  w \big| }_{=I_1}
+\underbrace{\big|\int_{\Omega\setminus \Omega'} \langle F, \nabla (\eta^2 w)\rangle \big|}_{=I_2} + C\int_{\Omega\setminus \Omega'}|w|^2 +\int_{\Omega\setminus \Omega'}| f \eta^2 w |$$
We can bound the  first term on the right hand side  by  
$$
I_1 \leq \big|2\int_{\Omega\setminus \Omega'}\eta \langle \sigma \nabla_A w,
 \nabla_A w\rangle^{1/2} \langle \sigma \nabla \eta, \nabla \eta\rangle^{1/2} |w| \big|
$$
$$\leq  2K\|\nabla \eta\|_{L^\infty}\left(\int_{\Omega\setminus \Omega'}\ \eta^2 \langle \sigma \nabla_A w, \nabla_A w\rangle\right)^{1/2}\left(\int_{\Omega\setminus \Omega'}\ |w|^2\right)^{1/2} $$
$$\leq  1/2 \left(\int_{\Omega\setminus \Omega'}\ \eta^2 \langle \sigma \nabla_A w, \nabla_A w\rangle\right) + 16K^2\|\nabla \eta\|_{L^\infty}^2 \int_{\Omega\setminus \Omega'}\ |w|^2,$$
where we have used that $\sigma \le KI$.
Hence we absorb the term $1/2 \left(\int \eta^2 \langle \sigma \nabla_A w, \nabla_A w\rangle\right)$ by the left hand side to obtain
\begin{equation}\label{homogeneousterm}
\begin{split}  |\int_{\Omega\setminus \Omega'}\eta^2\langle \sigma \nabla_A w, \nabla _A w\rangle | & \leq C \|\nabla \eta\|_{L^\infty}^2\int_{\Omega\setminus \Omega'}\ |w|^2 \\ &+2|\int_{\Omega\setminus \Omega'}\ \langle F , \nabla (\eta^2 w)\rangle| 
 +C\int_{\Omega\setminus \Omega'} |fw|.
\end{split}
\end{equation}
Now we deal with the term $I_2=|\int_{\Omega\setminus \Omega'}\ \langle F, \nabla (\eta^2 w)\rangle|.$  By integrating by parts
and the definition of $\nabla_A$ we have that,
$$I_2   \leq|\int_{\Omega\setminus \Omega'} \langle F,2\eta \nabla \eta\rangle w|+|\int_{\Omega\setminus \Omega'} \eta^2\langle F, \nabla_A  w\rangle |+\int_{\Omega\setminus \Omega'} |\langle F,\eta^2 iA   w\rangle|. $$
Next, we use the  Cauchy-Schwartz inequality for the first  and the third terms of the right and the H\"older inequality for the second, 
in order to bound $I_2$ by  
$$\leq   C(\int_{\Omega\setminus \Omega'} | F \eta|^2 +\int_{\Omega\setminus \Omega'} |\nabla \eta|^2 |w|^2 )
 +   \int_{\Omega\setminus \Omega'} \eta^2 | F|| \nabla_A w| + 
 \int_{\Omega\setminus \Omega'}  |F|^2\eta^2 + 
 \int_{\Omega\setminus \Omega'} |Aw|^2$$
$$
\leq C  \int_{\Omega\setminus \Omega'} \eta^2 |  F |^2 +C \|\nabla \eta\|_{L^\infty}^2 \int_{\Omega\setminus \Omega'}|w|^2 +
(\int_{\Omega\setminus \Omega'} \eta^2 | F |^2)^{1/2}
(\int_{\Omega\setminus \Omega'} \eta^2 | \nabla_A w |^2)^{1/2}.$$
Since,
$$(\int_{\Omega\setminus \Omega'} \eta^2 | F|^2)^{1/2}
(\int_{\Omega\setminus \Omega'} \eta^2 | \nabla_A w |^2)^{1/2}\leq  
C\int_{\Omega\setminus \Omega'} \eta^2 | F|^2 +
\frac 1 {2K}\int_{\Omega\setminus \Omega'} \eta^2 | \nabla_A w |^2,$$
 we have obtained the bound 
\begin{equation}\label{fterm}\begin{split}
I_2 \le   C\int_{\Omega\setminus \Omega'} \eta^2 | F |^2 +
\frac 1{2K}\int_{\Omega\setminus \Omega'} \eta^2 | \nabla_A w |^2  +C\|\nabla \eta\|_{L^\infty}^2\int_{\Omega\setminus \Omega'}|w|^2.
\end{split}
\end{equation}

Now we incorporate estimate (\ref{fterm}) into (\ref{homogeneousterm}) estimating the  term  $\int_{\Omega\setminus \Omega'}| f \eta^2 w |$ by 
Cauchy-Schwarz. We obtain that 
$$
|\int_{\Omega\setminus \Omega'}\eta^2\langle \sigma \nabla_A w, \nabla _A w\rangle |$$
$$\leq C\int_{\Omega\setminus \Omega'} \eta^2 | F |^2 +\frac 1{2K}\int_{\Omega\setminus \Omega'} \eta^2 | \nabla_A w |^2  +C\|\nabla \eta\|_{L^\infty}^2\int |w|^2 +\int |f|^2.$$
Since $\frac{1}{K} I \le \sigma $, 
$$ \frac1 K \int_{\Omega\setminus \Omega'} |\eta\nabla_A w|^2  \leq |\int_{\Omega\setminus \Omega'}\eta^2\langle \sigma \nabla_Aw, \nabla w\rangle |$$ 
and thus  we can absorb the term $\frac 1{2K}\int_{\Omega\setminus \Omega'} \eta^2 | \nabla_A w |^2$  to the left hand side to obtain the bound 

\begin{equation} \label{SuperCacciopoli}\begin{aligned}
\int_{\Omega\setminus \Omega'} |\eta\nabla_A w|^2 & \leq
C\left(\int_{\Omega\setminus \Omega'}   | F|^2 + \|\nabla \eta\|_{L^\infty}^2
\int_{\Omega\setminus \Omega'}  |w|^2  +
\int_{\Omega\setminus \Omega'} |f|^2 \right), \end{aligned} \end{equation}
where $C$ depends on $(K,A,q)$ but not of $\Omega',\Omega''$.

In order to obtain (\ref{equationring}) we simply  expand $\nabla_A w$ and observe that  $\eta=1$ on $\Omega \setminus \Omega''$.

In order to obtain (\ref{equationring2}) we define the cut-off more carefully. Let 
$\eta_\delta \in C^{0,1} (\Omega)$ be defined by   

\begin{equation}\label{etadelta}
\eta_\delta(x)= \eta(d(x, \p \Omega)/\delta),
\end{equation}
 where  $\eta(s) \in C^\infty_0(\R_+),\, \eta(s)=1$ for $s <1,\, \eta(s)=0$
 for $s>2$.
Observe that
$\supp(\eta_\delta) \subset \Omega_{2\delta}$  and  
\begin{equation}\label{delta}
  \|\eta_\delta \|_{C^{0,1}(\Omega)}  \le C \delta^{-1}.
\end{equation}

Plugging $\eta_\delta$ into (\ref{SuperCacciopoli}) yields (\ref{equationring2}).
\end{proof}

Let us fix a boundary value 
$\psi \in \H(\partial \Omega)$ and, as in the beginning of this section, 
denote the corresponding solutions   to (\ref{condmagn})  by $u_n^\psi$, 
with   $ u^\psi$ being the solution to (\ref{condmagn}) for $\L^\l$. 
 It will be convenient for us 
to work with the difference 
$$d^\psi_n(\lambda)=d^\psi_n =u^\psi_n-u^\psi  \in H^{1}_0(\Omega).$$
Due to (\ref{1.1.03}) and (\ref{2.1.03}), it follows that 
\[ \|d^\psi_n\|_{H^1(\Omega)} \le \|u^\psi_n\|_{H^1(\Omega)}+\|u^\psi\|_{H^1(\Omega)} \le 
C\|\psi\|_{\H(\partial \Omega)} \]
It is convenient  to state the above inequality as a separate lemma.

 \begin{Lemma} \label{D} 
Let
\begin{equation}\label{A}
{\mathcal A}^\lambda_n(\psi)={ d_n^\psi}_{|_ {\Omega\setminus\Omega'}}, \quad
{\mathcal A}^\lambda_n: H^{1/2}(\partial\Omega) \to 
\mathring {H}^{1} (\Omega\setminus\Omega').
\end{equation}
 Then these  operators  are uniformly bounded  wrt $n$ and $\l \in \Bbb K$.
 \end{Lemma}

 We prove now  the strong convergence of the Dirichlet to Neumann  maps.
\begin{Prop}  \label{key2.a} Let $\L_n,\, \L$ and $\Bbb K$ satisfy conditions
of Theorem \ref{2.1}.
Then, for any $\psi \in \H(\p \Omega)$,
\[ \lim_{n \to \infty} \|(\Lambda^\lambda _n -
\Lambda^\lambda  ) (\psi) \|_{H^{ -1/2}}=0 ,  \]
 the convergence being uniform for $\lambda \in \Bbb K$.
\end{Prop}

\begin{proof} Let us fix a boundary value 
$\psi \in \H(\partial \Omega)$ and define $u_n^\psi, u^\psi,
d_n^\psi={\mathcal A}^\l_n \psi$ as above. Observe that, with 
${\tilde \psi} \in H^1(\Omega),\, 
{\tilde \psi} = 0$ in $\Omega'$, $\tilde \psi|_{\p \Omega}=\psi$, 
$$
u_n^\psi={\tilde \psi}+ (\L_n- \l I)^{-1} F,\quad 
u^\psi={\tilde \psi}+ (\L- \l I)^{-1} F,
$$
where 
$$
F=\nabla_A \cdot \sigma \nabla_A {\tilde \psi}- (q-\l) {\tilde \psi} \in
H^{-1}(\Omega),\quad \hbox{supp}(F) \subset \Omega \setminus \Omega'.
$$
Then it follows from G-convergence that 
that  
\[ d_n^\psi \rightarrow 0,  \]
where convergence is weak in ${\mathring{H}}^{1}(\Omega\setminus \Omega')$ and strong in $L^2(\Omega\setminus \Omega')$. 
(A direct proof under condition~\ref{sharp_0}   is given in Lemma~\ref{lem:d_n}). We continue by applying  Caccioppoli inequality
(\ref{Caccioppoli}) and taking into the account that
\begin{equation} \label{1.3.03}
\nabla_A \cdot \sigma \nabla_A d_n^\psi- (q-\l) d_n^\psi =0 \quad \hbox{in}\,\,
\Omega \setminus \Omega',
\end{equation}
we see that 
$$
d_n^\psi|_{\Omega \setminus \Omega''} \to 0 \quad \hbox{in}\,\,
{\mathring{H}}^1(\Omega \setminus \Omega'').
$$
This implies the desired result taking into the account the weak
definition of the Dirichlet-to-Neumann map (\ref{weak formulation})
and the ability to take $\tilde \phi$ there so that
$
 {\tilde \phi}=0$ in $  \Omega''$  and
$$\|{\tilde \phi}\|_{H^1(\Omega \setminus \Omega'')} \leq C
\|{\phi}\|_{\H(\p\Omega)}.
$$
\end{proof}

In order to utilize that   the functions $d_n^\psi$ satisfy (\ref{1.3.03}), we introduce the following subspace:

\begin{Def} \label{L_s}
We denote    $L^2_s(\Omega\setminus\Omega')$  to be the  $L^2(\Omega\setminus\Omega')$-closure  of the  set 
$\{u:\, u\in  {\mathring {H}}_{loc}^{1}(\Omega\setminus\Omega') \text 
{ and }\L^\lambda  u=0 \,\, \text{in}\, \Omega\setminus\Omega'\}$.
\end{Def}

 \begin{Lemma}\label{solutions}
Let  $v\in L^2_s(\Omega\setminus\Omega')$. Then 
$v \in \mathring {H}_{loc}^{1} (\Omega\setminus\Omega')$ and  is  a solution 
in $\Omega\setminus \Omega'$ of equation (\ref{1.3.03}).
\end{Lemma}

\begin{proof}
  Let 
  $v_k \subset  \mathring {H}_\textrm{loc}^{1} (\Omega\setminus\Omega') \cap 
  L^2_s(\Omega\setminus\Omega')$ satisfy (\ref{1.3.03})
  and 
   \[ \lim_{k \to \infty} \|v_k-v\|_{L^2(\Omega \setminus \Omega')}=0. \]
  Then, for any 
 $\Omega''$ such that $\Omega'\subset  \Omega''\subset \Omega$, $v_k$ 
 is a Cauchy sequence in $\mathring {H}^{1} (\Omega\setminus\Omega'')$. Indeed,
 since $v_k$ vanish on $\partial\Omega,$ this follows from
  Lemma \ref{Caccioppoli}.  Thus,
  $v_k\to v$ strongly   in $\mathring {H}_{loc}^{1} (\Omega\setminus\Omega')$.
  As a strong limit of solutions is a weak solution to (\ref{1.3.03}),
   the claim follows.
 \end{proof}

\begin{Lemma} Fix $\lambda \in \Bbb K$.  Let ${\mathcal A}^\lambda_n$ 
 be defined by (\ref{A})  
and 
${\mathcal T}^\lambda: L^2_s(\Omega\setminus\Omega')\to H^{-1/2}(\partial \Omega)$ be defined for $\varphi\in H^{1/2}(\partial \Omega)$  as 
\begin{equation}\label{T}
({\mathcal T}^\lambda(v),\varphi)=\int (\langle \sigma\nabla_A v, \nabla_A\tilde\varphi\rangle+ (q-\lambda)v\overline {\tilde\varphi}),
\end{equation} 
where $ \tilde\varphi$ is  a $H^{1}(\Omega)$ extension of $\varphi$ such that  $\tilde\varphi=0$ on $\Omega''$.
Then 
\[\Lambda^\lambda _{ n}-\Lambda^\lambda = {\mathcal T^\l} \circ {\mathcal A}^\lambda_n. \]
Moreover, the operators ${\mathcal A}^\lambda_n$ are uniformly bounded and the operator 
${\mathcal T}^\lambda$ is compact.
\end{Lemma}

\begin{proof}

Notice that    the composition makes sense since 
$$\hbox{Range}({\mathcal A}^\lambda_n)\subset  L^2_s(\Omega\setminus\Omega')
\subset {\mathring {H}}_{loc}^1(\Omega \setminus \Omega').
$$  
The factorization is obvious 
and the uniform boundedness of ${\mathcal A}^\lambda_n$ is proven in  Lemma~\ref{D}.
Let us show that ${\mathcal T}^\lambda$ is compact. To this end it is  
sufficient to show that, if  $v_k\in L^2_s(\Omega\setminus\Omega')$ 
is a bounded sequence,  then  ${\mathcal T}^\lambda(v_k)$ is 
precompact in $H^{-1/2}(\partial \Omega)$. 

Let us take a  sequence of nested compact sets  
$\Omega' \subset \Omega'' \subset \Omega''' \subset \Omega$. 
 
It follows from Definition \ref{L_s} together with
 Cacciopoli inequality (\ref{Caccioppoli})   that  
$$
\|v_k\|_{H^{1} (\Omega\setminus\Omega'')} \le C \|v_k\|_{L^2 (\Omega\setminus\Omega')}.
$$
 By Banach-Alaoglu  theorem there is  a (not relabeled) subsequence $v_k$ 
 such that
 \[ v_k \to v_\infty \in H^{1} (\Omega\setminus\Omega'')
  \to 0, \]
where the convergence is
 weak  in $H^{1}(\Omega\setminus\Omega'')$ 
 and strong in ${L^2(\Omega \setminus \Omega'')}$.
 Then $v_\infty$
 is also a solution to the equation 
\[ \L^\lambda v_\infty =0 \textrm{ on }  \Omega\setminus\Omega''. \]
Thus, it follows  by   Caccioppoli inequality (\ref{Caccioppoli}) that
\begin{equation}\label{W2} \|v_k-v_\infty\|_{{H}^{1}(\Omega \setminus \Omega''')} \to 0. \end{equation}
Now, by the  definition of ${\mathcal T}^\lambda$,
$$\|{\mathcal T}^\lambda v_k-{\mathcal T}^\lambda v_\infty\|_{H^{-1/2}}= \sup_{\{\|\varphi\|_{H^{1/2}(\partial\Omega)}=1\}}\int( \langle \sigma \nabla_A (v_k-v_\infty), \nabla_A\tilde\varphi\rangle+(q-\lambda)(v_k-v_\infty)\overline {\tilde\phi}),
$$
where we take the  extension  function $\tilde\varphi$ so that 
$\hbox{supp}(\tilde\varphi) \subset \Omega\setminus\Omega'''$. Choosing
$\tilde\varphi$ so that
$$
\|\tilde\varphi\|_{H^{1}(\Omega)} \leq C \|  \varphi\|_{H^{1/2}(\p \Omega)},
$$
we see that
$$
\|{\mathcal T}^\lambda v_k-{\mathcal T}^\lambda v_\infty\|_{H^{-1/2}}
\leq CK
\|v_k-v_\infty\|_{H^{1}(\Omega \setminus \Omega''')}\|  \varphi\|_{H^{1/2}},$$
which tends to zero by (\ref{W2}). Thus, 
the desired compactness of ${\mathcal T}^\lambda(v_k)$ is proved. 

\end{proof}

We are now in position to complete the proof of Theorem \ref{2.1}.

\begin{proof}
Notice that $(H^{1/2})^*(\partial \Omega)=H^{-1/2}(\partial \Omega)$.
Taking $\tilde \phi$ in (\ref{weak formulation}) to be the solution of 
(\ref{condmagn}) with $\overline \l$ instead of $\l$ and $\phi$ instead 
of $\psi$, we see that 
$ (\Lambda^\l_{ n}-\Lambda^\l ):\H(\partial \Omega) \to 
H^{-1/2}(\partial \Omega) $ satisfies
$$
(\Lambda^\l_{ n}-\Lambda^\l )^*=(\Lambda^{\bar \l}_{ n}-\Lambda^{\bar \l} ).
$$
 Thus,
\[ (\Lambda^\l_{ n}-\Lambda^\l )=
\left({\mathcal T}^{\bar \l} \circ {\mathcal A}^{\bar\lambda}_n\right)^* = 
({\mathcal A}^{\bar\lambda}_n)^* \circ ({\mathcal T}^{\bar \l})^*, \]
where 
  $({\mathcal T}^{\bar \l})^*:H^{1/2}(\partial \Omega) \to L^2_s(\Omega \setminus \Omega')$  is a compact operator. Thus, for every 
$\epsilon>0$, there exists a finite dimensional projection operator $P_\epsilon: H^{1/2}(\partial \Omega) \to H^{1/2}(\partial \Omega)$, such that 
\begin{equation} \label{3.3.03}
\| ({\mathcal T}^{\bar \l})^*(I-P_\epsilon) \|_{H^{1/2}(\partial \Omega) \to L^2_s(\Omega \setminus \Omega')}\le \epsilon.
\end{equation}
Since $P_\epsilon$ is finite dimensional, it follows from the strong convergence of  $\Lambda^\l_{ n}-\Lambda^\l $, Proposition~\ref{key2.a},   that 
$$ 
\lim_{n \to \infty}\| \Lambda_{ n}^\l-\Lambda^\l )P_\epsilon \|_{H^{1/2}(\partial \Omega) \to H^{-1/2}(\partial \Omega) }=0.
$$
Moreover, the above limit is uniform wrt $\l \in \Bbb K$.
On the other hand, since 
$\| ({\mathcal A}^{\bar\lambda}_n)^* \|_{L^2_s(\Omega \setminus \Omega') \to H^{-1/2}(\partial \Omega) }\le C(\Bbb K),\, \l \in \Bbb K, $ we obtain 
from (\ref{3.3.03})
that 
$$
\| (\Lambda^\l_{ n}-\Lambda^\l ) (I- P_\e) \|_{H^{1/2}(\partial \Omega) \to H^{-1/2}(\partial \Omega) } \le 
  C \epsilon.
$$
Since $\epsilon$ is arbitrary, these estimates prove the theorem.
\end{proof}
\begin{Remark} \label{extra_smooth}
Assuming $\sigma, A \in C^{0,1}(\Omega \setminus \Omega')$, equation (\ref{28.3.1}) remains valid for the operator
norm in $H^{1/2}(\p \Omega)$. Moreover, assuming further smoothness of $\sigma$,  $A$ and $q$,
we obtain equation (\ref{28.3.1}) with the operator norm from $H^{1/2}$ to $H^s$ with
larger $s$.
\end{Remark}
\end{section}

\begin{section}{Application to cloaking} 
In this section we apply the previous construction to study an approximate invisibility as introduced in
\cite{GKLU}. To start, we recall the main result in \cite{GKLU}.

Let us consider $\Omega=B_r ,\ r=3$, where $B_r \subset \R^3$ is a ball of radius $r$ centered at $0$.
Denote by $\L_{out}$ the operator 
\begin{equation} \label{L-out}
\begin{aligned}
& \L_{out} u= -\nabla_{A_{out}} \cdot \nabla_{A_{out}} u +q_{out} u,
\\
&{\mathcal D}(A_{out})= \{u \in H^1_0(\Omega):\, \nabla_{A_{out}} \cdot \nabla_{A_{out}} u \in L^2(\Omega)\}.
\end{aligned}
\end{equation}
Here the magnetic potential $A_{out}$ and electric potential $q_{out}$ satisfy,
$$
|A_{out}|\cdot |x|\in L^{\infty}(\Omega), \quad q_{out} \in L^{\infty}(\Omega),
$$
see (15) and preceeding discussion in \cite{GKLU}, where $\beta_1$ stands for $A_{out}$ and $\kappa_1$
stands for $q_{out} $. Denote by $\Lambda_{out}^\l$ the Dirichlet-to-Neumann map corresponding
to $\L_{out}-\l$.

Next, consider the Dirichlet-to-Neumann map $\Lambda_{R, m,\e}^\l$ associated to the approximate cloaking.
To this end, consider the Dirichlet problem of the type (\ref{condmagn}),
\begin{equation}\label{D_invisibilty}
\begin{cases}
\L_{R, m, \e}^\lambda u=  -g_m^{-1/2}\nabla_{A} \cdot  \sigma_{R, \e}\nabla_{A} u +q u -\l u=0\\
u\big|_{\partial \Omega}=\psi,
\end{cases}
\end{equation}
cf. (126), (127) in \cite{GKLU}. Here $\sigma_{R, \e}$ is a regular, isotropic $G-$approximation to the singular cloaking
conductivity $\sigma_s$ and $g_m$ is a truncation of $g_s,$ namely, $ g_m(x)= \max\{m^{-1},\, g_s(x)\}$,
where $g_s= (\det \sigma_s)^{2/(n-2)}$.
To define  $\sigma_R, \,R  \geq 1$,
we start with the diffeomorphism
$F_R=(F_{1,R},  F_{2, R}):\,(B_3 \setminus B_{\rho})\sqcup B_R \to \Omega$, where $F_{2, R}$  is the identity on $B_R$,
while 
$$
 F_{1, R}(x)= \left(\frac{|x|} {2}+1\right) \frac{x}{|x|}, \,\, \rho < |x| <2, \, \rho=2(R-1);
 \quad F_{1, R}(x)=x,\,\, |x| >2.
$$
Then $\sigma_R=(F_R)_* (\gamma_0,\,  \gamma_0)$, where $\gamma_0^{ij}=\delta^{ij}$ so that $\sigma_R$ 
degenerates on $\p B_1$ when $R=1$ but is bounded for $R>1$, with however lower bound going to $0$
if $R \to 1$. We note that in \cite{GKLU}, for technical reasons, $\gamma_0$ is substituted by $2\gamma_0$ in $B_R$,
however, the constructions in \cite{GKLU} can be readily modified for the considered case.

With $A \in L^\infty(\Omega),\, q\in L^\infty(\Omega)$, the operator $\L_R$ is defined as in (\ref{operator-L})
with $\sigma_R$ instead of $\sigma_n$ and
an extra factor $g_s^{-1/2}$ in front of the main term in the right-hand side
of (\ref{operator-L}).  The operator $\L_1$ represents perfect cloaking but it is singular. To avoid further confusing in terminology
 we will denote this operator by $\L_{sing}$. The
operators $\L_R$ are self-adjoint in $L^2(\Omega, g_s^{1/2}dx) $.
Note that then
$$
A_{out} =(F_{1,1})^* (A|_{B_3 \setminus B_1}), \quad q_{out} =(F_{1,1})^* (q|_{B_3 \setminus B_1}),
$$
which, in particular, produces the $1/|x|$ singularity of $A_{out}$.

 At last, the isotropic $\sigma_{R, \e}$ are obtained from
$\sigma_{R}$ by de-homogenization, see S.3, \cite{GKLU}, so that, if $\l \notin \hbox{spec}(\L_{sing})$,
then, for $f \in L^2(\Omega)$,
$$
\left(\L_{R, \e, m}-\l I\right)^{-1} f \to \left(\L_{R, m}-\l I\right)^{-1} f,
$$
see Lemma 3.3, \cite{GKLU}.
Note that the condition $\l \notin \hbox{spec}(\L_{sing})$ implies that, for $R$
close to $1$, large $m$ and small $\e$, $\l$ is outside the spectra of all
the operators considered above so all the objects are well-defined.
 Then it is shown in \cite{GKLU}, see Corollary 4.4, that there exists
a sequence $R(n) \to 1,\, m(n) \to \infty,\, \e(n) \to 0$ such that, for any 
$h \in H^{3/2}(\p \Omega)$,
$$
\|\Lambda_{R(n), m(n),\e(n)}^\l h - \Lambda_{out}^\l h\|_{H^{1/2}(\p \Omega)} \to 0.
$$
Here $\Lambda_{R(n), m(n),\e(n)}^\l,\, \Lambda^\l_{out}$ are Dirichlet-to-Neumann maps
associated with the operators $\L_{R(n), \e(n), m(n)},\, \L_{out}$
and $\l \notin \hbox{spec}(\L_{sing})$.

Using the methods of section 2, we have
\begin{Theorem} \label{invisibility}
The exists a sequence $R(n) \to 1,\, m(n) \to \infty,\, \e(n) \to 0$ such that,
for any $\l \notin \hbox{spec}(\L_{sing})$,
$$
\|\Lambda_{R(n), m(n),\e(n)}^\l-\Lambda_{out}^\l \|_{H^{1/2}(\p \Omega) \to H^{-1/2}(\p \Omega)}
\to 0, \quad \hbox{as}\,\, n \to \infty.
$$
\end{Theorem}
Observe that $\Lambda_{out}$ does not depend upon the behaviour of $A$ and $q$ inside $B_1$. Thus, Theorem \ref{invisibility}
 means that $A|_{B_1},\, q|_{B_1}$ are almost cloaked from an external observer by a proper
choice of $\sigma_{R, \e}$. 

\begin{proof}
By Theorem 4.3, \cite{GKLU}, for $f \in L^2(\Omega)$ and $ \l \notin \hbox{spec}(\L_{sing})$,
\begin{equation} \label{r,e,m-limit}
\lim_{n \to \infty}\left(\L_{R(n), \e(n), m(n)} -\l I  \right)^{-1} f =
\left(\L_{sing} -\l I  \right)^{-1} f \quad \hbox{in} \,\, 
L^2(\Omega, g_s^{1/2} dx),
\end{equation}
where the convergence is uniform for $\l \in \Bbb K$, $\Bbb K$ being a compact in 
$\C \setminus \hbox{spec}(\L_{sing})$. 

Let, for $f \in L^2(\Omega \setminus B_2)$
\begin{equation} \label{mathcal_R}
{\mathcal R}_{R, m, \e}(\l) f= \left((\L_{R, m ,\e} -\l I )^{-1} f \right)|_{\Omega \setminus B_2} \in 
\mathring{H}^{1}(\Omega\setminus B_2) ,
\end{equation}
where in the right-hand side we continue $f$ by $0$ to $B_2$  and we use similar notation for 
${\mathcal R}_{R, m}(\l) $, etc. 
Our next goal is to show that, for $f \in L^2(\Omega \setminus B_2)$ and $\l \in \Bbb K$,
\begin{equation} \label{R1}
{\mathcal R}_{n}(\l) f \to {\mathcal R}_{sing}(\l) f, \quad \| {\mathcal R}_{n}(\l)  \|_{L^2 \to \mathring{H}^{1}} < C(\Bbb K).
\end{equation}
Here ${\mathcal R}_{n}(\l) ={\mathcal R}_{R(n), m(n), \e(n)}(\l) $ and the convergence in (\ref{R1})
is the weak  convergence in $\mathring{H}^{1}(\Omega\setminus B_2) $.

Indeed, using Lemmata 2.7, 2.8, \cite{GKLU}, for
$f \in L^2(\Omega, g_s^{1/2} dx)$ and  $\l \in \Bbb K$,
$$
\lim_{n \to \infty}\left(\L_{R(n)} -\l I  \right)^{-1} f =
\left(\L_{sing} -\l I  \right)^{-1} f \quad \hbox{in} \,\, H^1_0(\Omega, g_s^{1/2} dx),
$$
and there are $C(\Bbb K), R(\Bbb K)>1$ such that, for $R < R(\Bbb K)$, 
$$
\|\left(\L_{R} -\l I  \right)^{-1}\|_{L^2(g_s^{1/2} dx) \to H^1_0(g_s^{1/2} dx)} <C(\Bbb K).
$$

Since $\sigma_R(x)=\gamma_0,
\, g(x)=1$ for $|x| >2$, 
 these two equations imply  that  
\begin{equation} \label{step-1}
{\mathcal R}_R(\l) f \to {\mathcal R}_{sing}(\l) f \,\, \hbox{in}\,\, \mathring{H}^{1}(\Omega\setminus B_2),
\quad \|{\mathcal R}_R(\l)\|_{L^2 \to \mathring{H}^{1}}  <C(\Bbb K).
\end{equation}
Next, using Lemma 2.11, \cite{GKLU}, we see that equation (\ref{step-1})
remains valid if we put in  ${\mathcal R}_{R, m}(\l)$ instead of ${\mathcal R}_R(\l)$ and first take the limit as $m \to \infty$ and then 
as $R \to 1$. Here ${\mathcal R}_{R, m}(\l)$
are defined  by (\ref{mathcal_R}) with $\L_{R, m}$. 

At last, by means of Lemma 3.3, \cite{GKLU},  we see that (\ref{step-1})
remains valid, in the sense of the weak-convergence, if we
 put 
${\mathcal R}_{R, m, \e}(\l)$ instead of ${\mathcal R}_{R, m}(\l)$ and (\ref{R1}) follows.

Since ${\mathcal R}_{R, \e, m}^*(\l) h={\mathcal R}_{R, \e, m}({\overline \l}) h$, if
$h \in L^2(\Omega \setminus B_2)$,
 it follows from (\ref{R1}) that, for 
$h \in \left(\mathring{H}^{1}(\Omega\setminus B_2)\right)^*$ and $\l \in K$,
$$
\lim_{R \to 1} \lim_{m \to \infty} \lim_{\e \to 0}{\mathcal R}_{R, \e, m}(\l) h \to {\mathcal R}_1(\l) h \,\, \hbox{in}\,\, L^2(\Omega\setminus B_2),
$$
in the sense of the weak convergence in $L^2(\Omega\setminus B_2)$
and, when $R$ is sufficiently close to $1$, $m$ is sufficiently large and $\e$ 
is sufficiently close to $0$,
\begin{equation} \label{step-4}
 \|{\mathcal R}_n(\l)\|_{(\mathring{H}^{1}(\Omega\setminus B_2))^* \to L^2(\Omega\setminus B_2)} <C(\Bbb K),
\end{equation} 
Then, similar to the proof of Theorem 4.3,
\cite{GKLU}, there is a sequence $R(n) \to 1, \, m(n) \to
\infty,\, \e(n) \to 0$, such that
\begin{equation} \label{step-3}
\langle{\mathcal R}_n(\l) h,\, f \rangle \to 
\langle{\mathcal R}_1(\l)h, \, f \rangle,\quad
h \in \left(\mathring{H}^{1}(\Omega\setminus B_2)\right)^*,\,\,
f \in L^2(\Omega\setminus B_2).
\end{equation}
 Moreover, operators
${\mathcal R}_n(\l)$ satisfy  (\ref{step-4}).

To continue, consider the solutions $u_n^\psi(\l)$ and $u_1^\psi(\l)$ to
the Dirichlet problems (\ref{D_invisibilty}) for $\L_n^{\l}$ and $\L_{sing}^{\l}$
 with $\psi \in H^{1/2}(\p \Omega)$. 
Using a bounded extension,
$\tilde \psi \in H^1(\Omega),\, \hbox{supp}(\tilde \psi) \subset \Omega \setminus B_{5/2}$,
we see that $d_n^\psi(\l)= \left(u_n^\psi(\l)-u_{sing}^\psi(\l)\right)|_{\Omega \setminus B_2}$
satisfies
$$
d_n^\psi(\l) = \left({\mathcal R}_n(\l)-{\mathcal R}_{sing}(\l)  \right) h,
$$
where $h$ is given in terms of the  extension as
$$
h= \nabla_{A_{out}} \cdot \nabla_{A_{out}} \tilde \psi -(q_{out} -\l) \tilde \psi.
$$
Note that since $\hbox{supp}(h) \subset \Omega \setminus B_{5/2},$ we have
$h \in (\mathring{H}^{1}(\Omega \setminus B_2)^*$. Moreover, 
\begin{equation}\label{weak}
\|d_n^\psi(\l)\|_{L^2(\Omega \setminus B_2)} <C(\Bbb K) \|\psi\|_{H^{1/2}(\p \Omega)},\quad 
w-\lim d_n^\psi(\l) = 0. 
\end{equation}
where $w-\lim$ is the weak limit in $L^2(\Omega \setminus B_2)$.

Now, notice that  $\L_{sing} ^{\lambda} (d^\psi_n(\lambda))=0$ on $\Omega \setminus B_2$. Thus, we can
use the Cacciopoli inequality (\ref{equationring}) to obtain that   
\begin{equation} \label{3.9}
\|d_n^\psi(\l)\|_{H^1(\Omega \setminus B_{5/2})} \le  \|d_n^\psi(\l)\|_{L^2(\Omega \setminus B_2)} <C(\Bbb K) \|\psi\|_{H^{1/2}(\p \Omega)}
\end{equation}
and thus, by compactness of 
the Sobolev embeddding and (\ref{weak}) (and Kuratowski-Zorn Lemma), 
it follows that 
\begin{equation} \label{limits}
\limsup_{n \to \infty} \|d_n^\psi(\l)\|_{H^1(\Omega \setminus B_{11/4})} 
\le c \lim_{n \to \infty} \|d_n^\psi(\l)\|_{L^2(\Omega \setminus B_{5/2})} =0 \end{equation}
We can now mimic the arguments in section 2.  Namely, 
recall that, for $\psi, \varphi \in \H(\partial \Omega)$, we have that,
\[
\langle (\Lambda_{R(n), m(n),\e(n)}^\l-\Lambda_{out}^\l)\psi,
\varphi \rangle
\]
\[= 
\int_{\Omega \setminus B_{11/4}}
\left(\nabla_{A_{out}} d_n^\psi \cdot {\overline 
{\nabla_{A_{out}} {\tilde{\varphi}}}}
+(q_{out}-\l)d_n^\psi \bar{\tilde{\varphi}} \right).  \]
Here $\tilde{\varphi},\, \hbox{supp}(\tilde{\varphi}) \subset \Omega 
\setminus B_{11/4}$ is the extension of $\varphi$.
Thus, cf. the proof of Proposition \ref{key2.a} we have that
\begin{equation}\label{compactness}\begin{aligned}
\|(\Lambda_{R(n), m(n),\e(n)}^\l-\Lambda_{out}^\l) \psi\|_{ H^{-1/2}(\p \Omega)} \le C  
\|d_n^\psi(\l)\|_{H^1(\Omega \setminus B_{11/4})}
 \end{aligned}
 \end{equation}
and hence (\ref{limits}) yields the  strong convergence
\begin{equation}
\lim_{n \to \infty} \|(\Lambda_{R(n), m(n),\e(n)}^\l-\Lambda_{out}^\l) \psi\|_{ H^{-1/2}(\p \Omega)}=0
\end{equation}
Next, we  introduce the intermediate space 
\begin{Def}We denote    $L^2_s(\Omega\setminus B_{11/4})$  to be the  $L^2(\Omega\setminus B_{11/4})$-closure  of the  set $\{ u\in  {\mathring {H}}_{loc}^{1}(\Omega\setminus B_{11/4}):\,\L_{sing}^\lambda  u=0\}$.
\end{Def} 
We
factorize the difference of Dirichlet to Neumann maps by 
\[ \Lambda_{R(n), m(n),\e(n)}^\l-\Lambda_{out}^\l={\mathcal T}^\l \circ 
{\mathcal A}^\l_n. \]
Exactly as in the end of proof of Theorem \ref{2.1} 
in section 2,  ${\mathcal A}^\l_n: 
\H(\partial \Omega) \to  L^2_s(\Omega\setminus B_{11/4})$, defined by $A_n(\psi)=d^\psi_n(\lambda)$, is uniformly bounded 
in $n$ and, due to (\ref{3.9}), (\ref{compactness}),
 ${\mathcal T}^\l:L^2_s(\Omega\setminus B_{11/4}) \to  H^{-1/2}(\partial \Omega)$ are 
 compact. Since 
\[
\left(\Lambda_{R(n), m(n),\e(n)}^\l-\Lambda_{out}^\l\right)^*=
\Lambda_{R(n), m(n),\e(n)}^{\bar\l}-\Lambda_{out}^{\bar \l},
\] 
this gives rise to the factorization 
$({\mathcal A}^{\bar \l}_n)^* \circ ({\mathcal T}^{\bar\l})^*$ with compact
$({\mathcal T}^{\bar\l})^*$. Thus,
 we can find $P_\epsilon$ so that $({\mathcal T}^{\bar \l})^*(I-P_\epsilon)$ is small in norm
and prove the theorem.
\end{proof}

\begin{Remark}
By a slight modification of the arguments we can show that Theorem \ref{invisibility}
remains valid if $B_R$ is changed into an arbitrary smooth Riemannian manifold $(M, g)$
with $\p M$ diffeomorphic to $\p B_1$, cf. \cite{GKLU1}.
\end{Remark}

\begin{Remark} Similar to Remark \ref{extra_smooth}, Theorem \ref{invisibility}
remains valid
for the operator norm in $H^{1/2}(\p \Omega)$ if  $\sigma$ and $A$ are $C^{0, 1}-$smooth near $\p \Omega$.
\end{Remark}

\end{section}

\begin{section}
{\bf General Condition}

In this section we relax the conditions on the behaviour of 
$\sigma_n$ and $\sigma$ near the boundary under which 
the $G-$convergence implies the convergences of the Dirichlet-to-Neumann 
maps in
 the operator norm. 
 It will be desirable to be able to deal with the situation when,  for every $n$,
\[|\sigma_n(x)-\sigma(x)| \le C d(x,\partial \Omega)^{1+\epsilon}, \epsilon>0 \]
(this is the condition suggested by G.Alessandrini as mentioned in Introduction),
or when, for some $\Omega' \Subset \Omega$, we have that 
\[ \lim_{n \to \infty} \|\sigma_n-\sigma\|_{L^\infty(\Omega \setminus \Omega')}=0. \]
As discussed in the introduction,  we prove that actually a condition resembling the convergence of the conductivities and their normal derivatives at the boundary 
and   weaker than both  conditions above suffices.
\begin{Theorem}\label{general} Let for $\delta >0,\,$ 
 let $\Omega_{\delta}=\{x \in \Omega: d(x,\partial \Omega) \le \delta\}$. 
Assume that 
\begin{equation}\label{sharp}
\lim_{\delta \to 0} \delta^{-1} 
\left(\limsup_{n \to \infty} \|\sigma_n-\sigma\|_{L^{\infty}(\Omega_\delta)} \right)=0
\end{equation}
and that $\sigma_n \in M_K$ converges to $\sigma$ in the sense of the
$G-$convergence. Then 
\[\lim_{n \to \infty} \|\Lambda_{\sigma_n}-\Lambda_{\sigma} 
\|_{\H(\partial \Omega) \to H^{-1/2}(\partial \Omega)}=0.\]
\end{Theorem}

For the sake of simplicity we will consider only the case of the conductivity equation at $\l =0$, however,  with the obvious modifications
 the proof will work as well for the more general operators
treated in section 2.   We also recall the smooth extensions and restrictions of Sobolev functions to $\Omega_\delta$. 

Let $\eta_\delta \in C^{0,1}(\Omega)$ be supported in 
$\Omega_\delta$ with 
$\|\nabla \eta_{\delta} \|_{L^\infty(\Omega_{2\delta})} \le C/ \delta$, see (\ref{etadelta}) from the 
proof of \eqref{equationring}.

  Then  
 set 
 $\psi_\delta=\eta_\delta \tilde \psi \in H^1(\Omega)$ and 
 observe that
$\supp(\psi_\delta) \subset \Omega_{2\delta}$  and  
\begin{equation}\label{delta}
\| \psi_\delta \|_{H^1(\Omega)} \le \|\eta_\delta \|_{C^{0,1}(\Omega)} \|\tilde \psi\|_{H^1(\Omega)} \le C \delta^{-1} \|\psi\|_{H^{1/2}(\partial\Omega)}.
\end{equation}
Now recall Caccioppoli estimate \eqref{equationring2},

  \begin{equation}\label{1.7.03}
\int_{\Omega_\delta}|\nabla w|^2 
\leq C \left(\delta^{-2} \int_{\Omega_{2\delta}}|w|^2+
\int_{\Omega_{2\delta}}|F|^2 +
\int_{\Omega_{2\delta}}|f|^2 \right),
 \end{equation}
 where $w\in {\mathring H}^1(\Omega_{2\delta})$ 
 satisfies (\ref{3.1.03}) in $\Omega_{2\delta}$.

In order to prove the convergence 
of the Dirichlet-to-Neumann maps in the operator norm,  we treat 
the current case as a perturbation of the one considered in section 2,
where
$\sigma_n=\sigma$ near the boundary.  As in section 2, we introduce the function $d_n^{\psi}=u_n^\psi- u^\psi$. 

Let us recall that the $G-$convergence implies the convergence of the solutions to the corresponding Dirichlet problems 
  (see \cite[Thm22.9]{D}). In the next lemma we give a quick proof under the condition (\ref{sharp_0}) valid also for the situation in section 
  3 where we do not have uniform ellipticity. 
  
  
  
 \begin{Lemma} \label{lem:d_n} Let $\sigma_n$  G-converge to $\sigma$. Then,
for any $\psi \in \H(\p \Omega)$,
\begin{equation}\label{2.7.03}
\|d_n^{\psi} \|_{L^2(\Omega)} \to 0,\quad{as}\,\, n \to \infty.
\end{equation}
\end{Lemma}
\begin{proof}
With $\psi_\delta$ as above, we have
\[
\begin{aligned}
& \L_n (u_n^\psi-\psi_\delta)=\nabla \cdot \Psi_\delta +
\nabla \cdot F_{n, \delta},\quad
\L (u^\psi-\psi_\delta)=\nabla \cdot \Psi_\delta, \\
& \Psi_\delta = -\sigma \nabla \psi_\delta, \quad F_{n, \delta}=
(\sigma-\sigma_n)\nabla \psi_\delta.
\end{aligned}
\]
Thus,
\[
d_n^\psi= (\L_n^{-1}-\L^{-1}) (\nabla \cdot \Psi_\delta) + 
\L_n^{-1}(\nabla \cdot F_{n, \delta}) =I^1_{n, \delta}+I^2_{n, \delta}.
\]

Notice that in the case $\sigma=\sigma_n$ on $\Omega_\delta$,  $I^2_{n, \delta}=0$.
Due to (\ref{1.1.03}) and (\ref{delta}),
\[
\begin{aligned}&
\|I^2_{n, \delta} \|_{H^1_0(\Omega)} \leq C(K) 
\|\sigma_n-\sigma  \|_{L^\infty(\Omega_{2\delta})} 
\|\psi_\delta  \|_{H^1(\Omega)} \\
& \leq C(K) \delta^{-1}
\|\sigma_n-\sigma  \|_{L^\infty(\Omega_{2\delta})} \|\psi\|_{\H(\p \Omega)}.
\end{aligned}
\]
Thus, due to condition (\ref{sharp_0}), for any $\e >0$ there are
$\delta(\e),\, n(\delta,\e)$ such that if $\delta <\delta(\e),\,
n >n(\delta,\e)$, 
\[
\|I^2_{n, \delta} \|_{H^1_0(\Omega)} <\e.
\]
Fixing $\delta <\delta(\e)$ and taking into the account that
$\sigma_n$ $G-$converges to $\sigma$, we see that
\[
\|I^1_{n, \delta} \|_{L^2(\Omega)} \to 0 \quad \hbox{as}\,\, n \to \infty.
\]
These two equations imply (\ref{2.7.03}). 
\end{proof}

 Next we represent $d_n^\psi$ in $\Omega_{2\delta}$ as
 \[
 d_n^\psi=v_{n, \delta}+m_{n, \delta},
 \]
 where $v_{n, \delta}^\psi,\,m^\psi_{n,\delta}$ are the solutions to the 
following equations
\begin{equation}\label{vsol}
\begin{cases}
&\nabla \cdot \sigma \nabla v_{n,\delta}^{\psi}=
\nabla \cdot((\sigma -\sigma_n)\nabla u^\psi_n)\\
&v_{n,\delta}^{\psi}=0 \text{ on }\partial (\Omega_{2\delta})
\end{cases}
\end{equation}
and
\begin{equation}\label{msol}
\begin{cases}
&\nabla \cdot \sigma \nabla m_{n,\delta}^{\psi}=0\text{ in }\Omega_{2\delta}\\
&m_{n,\delta}^{\psi}=d_n^{\psi} \text{ on }\partial (\Omega_{2\delta}).
\end{cases}
\end{equation}
Notice that $v_{n,\delta}^{\psi} \in {H}_0^{1} (\Omega_{2\delta}),\,
m_{n,\delta}^{\psi}\in \mathring {H}^{1} (\Omega_{2\delta})$.

Therefore, using the weak definition of the Dirichlet-to-Neumann map, we have
\begin{equation}\label{decomposition}
\begin{aligned}
&\langle(\Lambda_{\sigma_n}- \Lambda_{\sigma})\psi,\varphi\rangle=
\int_{\Omega_{2\delta}}(\sigma_n\nabla u^\psi_n-\sigma\nabla u^\psi )\cdot \nabla \varphi_\delta
\\&=
\int_{\Omega_{2\delta}}\sigma\nabla d_n^\psi \cdot \nabla \varphi_\delta +
\int_{\Omega_{2\delta}}(\sigma_n-\sigma)\nabla u^\psi_h\cdot \nabla  \varphi_\delta
\\&=
\int_{\Omega_{2\delta}}(\sigma_n-\sigma)\nabla u^\psi_h\cdot \nabla \varphi_\delta +
\int_{\Omega_{2\delta}}\sigma\nabla v^\psi_{n, \delta} \cdot \nabla \varphi_\delta +
\int_{\Omega_{2\delta}}\sigma\nabla m^\psi_{n, \delta} \cdot \nabla \varphi_\delta
\\&=
\langle D_{n, \delta} \psi,\,\varphi\rangle+
\langle V_{n, \delta} \psi,\,\varphi\rangle+\langle M_{n, \delta} \psi,\,\varphi\rangle.
\end{aligned}
\end{equation}
We summarize the above in the following lemma.
\begin{Lemma}\label{tres}
Let ${D_{n,\delta}}, V_{n,\delta}, M_{n,\delta}:\,
\H(\partial \Omega) \to H^{-1/2}(\partial \Omega) $ be defined 
by the last equation in (\ref{decomposition}). Then, 
for each $\delta>0$,
\begin{equation} \label{5.8.03}
\Lambda_{\sigma_n}- \Lambda_{\sigma}=
{D_{n,\delta}}+ V_{n,\delta}+ M_{n,\delta}.
\end{equation}
\end{Lemma}

We start bounding  the  first  and  second terms on the above decomposition.
\begin{Lemma}\label{DV}
Let $D_{n,\delta}$ and ${V_{n,\delta}}$ be defined as above. Then 
\begin{equation} \label{DV-eq}
\begin{aligned}
&\|D_{n, \delta}\|_{H^{1/2}(\partial\Omega)\to H^{-1/2}(\partial\Omega)}+\|{V}_{n,\delta}\|_{H^{1/2}(\partial\Omega)\to H^{-1/2}(\partial\Omega)} 
\\&
\le C
\delta^{-1} \| \sigma_n-\sigma\|_{L^{\infty}(\Omega_{2\delta})} .
\end{aligned}
\end{equation}
\end{Lemma}

\begin{proof}
The case of $D_{n,\delta}$ follows from Cauchy-Schwartz inequality and (\ref{delta}).   
The definition of  $ V_{n,\delta}$ implies that  
\[ \|{V_{n,\delta}}\|_{H^{1/2}(\partial\Omega)\to H^{-1/2}(\partial\Omega)} \le C \delta^{-1} \,
 \sup_{ \|\psi\|_{\H(\partial \Omega)}=1}\, \|\nabla v^\psi_{n,\delta}\|_{L^2(\Omega_{2\delta})} \]
Using  $v^\psi_{n,\delta}\in  H^{1}_0(\Omega_{2\delta})$  as  a test function for  the  
equation \eqref{vsol}, we see that 
\begin{equation} \label{2.8.03}
\begin{aligned} 
&\int_{\Omega_{2\delta}}  |\nabla v^\psi_{n,\delta}|^2  \le K  
   \int_{\Omega_{2\delta}}|\sigma \nabla v^\psi_{n,\delta} \cdot \nabla  v^\psi_{n,\delta}| \\ &\leq K \int_{\Omega_{2\delta}} |(\sigma -\sigma_n)\nabla u^\psi_n\cdot \nabla  v^\psi_{n,\delta}|
\\ & 
\leq K \| (\sigma -\sigma_n)\nabla u_n^\psi\|_{L^2(\Omega_{2\delta})}\,\|\nabla v_n^\psi\|_{L^2(\Omega_{2\delta})}
\\ &  \le K \|\sigma -\sigma_n\|_{L^\infty(\Omega_{2\delta})}\,
 \| \nabla u_n^\psi\|_{L^2(\Omega_{2\delta})}\,
 \|\nabla v_n^\psi\|_{L^2(\Omega_{2\delta})}
 \end{aligned}
 \end{equation}
Dividing both terms by 
$\|\nabla v_{n, \delta}^\psi\|_{L^2(\Omega_{2 \delta})}$ and recalling that 
$\| \nabla u_n^\psi\|_{L^2(\Omega)} \le  C \|\psi\|_{\H(\partial \Omega)}$
the claim follows.
\end{proof}
Note that, by continuing $v^\psi_{n, \delta}$ by $0$ into $\Omega \setminus
\Omega_{2\delta}$, (\ref{2.8.03}) implies that
\begin{equation} \label{3.8.03}
\| v_{n, \delta}^\psi\|_{{\mathring{H}}^1(\Omega)} \le C \|\sigma -\sigma_n\|_{L^\infty(\Omega_{2\delta})}\,\|\psi\|_{\H(\partial \Omega)}.
\end{equation}

Now we deal with the term $M_{n,\delta}$.  Namely, we prove 
the strong uniform boundedness of both $M_{n,\delta}(\psi)$
and $M^*_{n,\delta}(\psi)$.  

 \begin{Lemma}\label{P} Let $\psi \in H^{1/2}(\partial\Omega)$ and 
 $\delta>0$. Then there exists a constant $C(K)$ such that ,
  for any $\delta >0$,
 \begin{equation} \begin{aligned} \label{1.9.03}
 \limsup_{n \to \infty} \|M_{n,\delta}(\psi)\|_{H^{-1/2}(\partial\Omega)} & \leq
 C \delta^{-1} \limsup_{n \to \infty}  
 \| \sigma_n-\sigma\|_{L^{\infty}(\Omega_{2\delta})}
 \|\psi\|_{H^{1/2}(\partial\Omega)}.
\end{aligned}
\end{equation}
Also,
 \begin{equation} \label{2.9.03}\begin{aligned}
 \limsup_{n \to\infty} \|M^*_{n,\delta}(\psi)\|_{H^{-1/2}(\partial\Omega)}
\leq C\delta^{-1} \limsup_{n \to \infty}  
 \| \sigma_n-\sigma\|_{L^{\infty}(\Omega_{2\delta})}  \|\psi\|_{H^{1/2}(\partial\Omega)} .
\end{aligned}
\end{equation}
 \end{Lemma}

 \begin{proof}
 Notice that, as follows from (\ref{delta}) and (\ref{decomposition}),  
for each $\delta>0 $ and $\varphi \in H^{1/2}(\partial\Omega)$,
 \begin{equation} \label{11.1.04}  
 |\langle M_{n,\delta}(\psi), \varphi\rangle| \le C 
  \delta^{-1} \|\nabla m^\psi_{n,\delta} \|_{L^2(\Omega_\delta)} 
  \| \varphi\|_{H^{1/2}(\partial\Omega)} 
\end{equation}
Now we use $m^\psi_{n,\delta}-d^{\psi}_n =-v^\psi_{n,\delta} \in H^1_0(\Omega_{2\delta})$ 
as a test function in 
equation (\ref{msol}). By the strong ellipticity we get
  that  
\[ 
\|\nabla m^\psi_{n,\delta} \|_{L^2(\Omega_{2\delta})} \le K^2 \|\nabla d^\psi_{n} \|_{L^2(\Omega_{2\delta})}
\leq C \| \psi \|_{H^{1/2}(\p \Omega)}.
\] 
 Now observe that 
\[
\begin{cases}
&\nabla \cdot \sigma_n \nabla  d_{n}^{\psi}= \nabla \cdot (\sigma_n-\sigma) \nabla  u^{\psi}\text{ in }\Omega_{2\delta}\\
&d_n^{\psi} =0 \text{ on }\partial \Omega.
\end{cases}
\]
Thus, by (\ref{1.7.03}) it follows that
\[
\begin{aligned}
\|\nabla d^\psi_{n} \|_{L^2(\Omega_\delta)}  & \le 
\frac{C}{\delta} \|d^\psi_{n}\|_{L^2(\Omega)}+ C \| (\sigma -\sigma_n)\nabla u^\psi \|_{L^2(\Omega_{2\delta})}
\\ & \le \frac{C}{\delta}\|d^\psi_{n}\|_{L^2(\Omega)}+ C 
\| (\sigma -\sigma_n)\|_{L^2(\Omega_{2\delta})} \|\nabla u^\psi \|_{L^2(\Omega)}.
\end{aligned}
\]
Fix $\delta >0$ and let $n$ go to $\infty$. Then, using
  Lemma \ref{lem:d_n}
  and (\ref{11.1.04}), we arrive at the desired estimate (\ref{1.9.03}).

   At last, (\ref{2.9.03}) follows from (\ref{1.9.03})
and (\ref{DV-eq}) 
since, due to the fact that the DN  maps  are  self-adjoint, 
we have  from (\ref{5.8.03}) that
\[ 
 M^*_{n,\delta} =M_{n,\delta}+
{V}_{n,\delta}+D_{n,\delta}-{V^*}_{n,\delta} -D^*_{n,\delta}.
\]
\end{proof}

Next, following section 2, we reintroduce the following definitions:
\begin{Def}
We denote   by $L^2_{s}(\Omega_{2\delta})$   the  $L^2(\Omega_{2\delta})$-closure 
 of the  set $\{ u\in  {\mathring {H}}_{loc}^{1}(\Omega_{2\delta}):\, 
 \nabla \cdot \sigma\nabla  u=0\}$.
 \end{Def}
 We introduce also the modified operators,
 \begin{equation} \label{6.8.03}
 {\mathcal A}_{n,\delta}: \H(\p \Omega) \to L^{2}_s(\Omega_{2\delta}),
 \quad
 {\mathcal A}_{n,\delta}(\psi)=  m_{n, \delta}^\psi,
 \end{equation}
 and  
 \begin{equation} \label{7.8.03}
 {\mathcal T}^\delta: 
 L^2_{s}(\Omega_{2\delta}) \to H^{-1/2}(\partial\Omega), \quad 
  \langle {\mathcal T}^\delta(v),\varphi \rangle
 =\int_{\Omega_{2\delta}} \sigma \nabla v \cdot \nabla \varphi_\delta,
 \end{equation}
 so that
\begin{equation} \label{4.9.03}
{M}_{n,\delta}={\mathcal T}^\delta \circ {\mathcal A}_{n,\delta}.
\end{equation}
 Here $\psi_\delta$ is defined as in the beginning of this section.

\begin{Lemma} \label{T-A}
The operators ${\mathcal T}^\delta$
are compact and, for any fixed $\delta >0$, the operators 
${\mathcal A}_{n,\delta}$ are uniformly bounded wrt $n$.
\end{Lemma}
\begin{proof}
As in section 2, we notice that the Caccioppoli inequality,
Lemma \ref{Caccioppoli} and the compactness of the Sobolev embedding 
imply that the  restriction operator,
\[
R_\delta:L^2_{s}(\Omega_{2\delta}) \to {\mathring{H}}^1(\Omega_{3\delta/2}) \to
L^2_s(\Omega_{3\delta/2}) \to {\mathring{H}}^1(\Omega_{\delta}),
\] 
is a bounded and, due to the second embedding above, compact
operator. Now the operator  $ {\widetilde{\mathcal T}^\delta}: H^1(\Omega_\delta) \to H^{-1/2}(\partial\Omega)$ 
defined by 
\[ \langle \tilde{T}_\delta(v),\varphi\rangle
=\int_{\Omega_\delta} \sigma \nabla v \cdot \nabla \varphi_{\delta/2},
\quad \varphi \in H^{1/2}(\p \Omega),
\]
is continuous. Hence ${\mathcal T}^\delta={\widetilde{\mathcal T}^\delta} \circ R_\delta$ is also compact.

As for the uniform boundedness, for a fixed $\delta$, of the
 operators ${\mathcal A}_{n,\delta}$, it follows from
the decomposition
of    $A_{n,\delta}$ in form:
$$\psi \in H^{1/2}\mapsto d_{n,\delta}^\psi \in H^1(\Omega) \mapsto
\hbox{trace} ( d_{n}^\psi )\in H^{1/2}(\partial\Omega_{2 \delta})\mapsto m_{n,\delta}^\psi \in L^{2}(\Omega_{2 \delta}).$$
\end{proof}
 We can now return to the uniform estimate for the operators
$ M_{n,\delta}$.
 \begin{Lemma}\label{Poperator}
  \begin{equation} \limsup_{n \to \infty} \|M_{n,\delta}\|_{H^{1/2}(\partial\Omega)\to H^{-1/2}(\partial\Omega)} \le C\delta^{-1} \limsup_{n \to \infty}  
 \| \sigma_n-\sigma\|_{L^{\infty}(\Omega_{2\delta})}. 
\end{equation}
 \end{Lemma}
 
 \begin{proof}
 By (\ref{4.9.03}), it follows from Lemma \ref{T-A}
that, for any $\e >0, \, \delta>0,$
there is a finite dimensional projector 
$P_{\epsilon, \delta}: \H(\partial \Omega) \to \H(\partial \Omega)$ such that 
\begin{equation}\label{1-proy} \| M_{n,\delta}^* \circ (I-P_{\epsilon,\delta}) \|_{\H(\partial \Omega) \to H^{-1/2}(\partial \Omega)} \le \epsilon.
\end{equation}
Then,   by using estimate (\ref{2.9.03}) and the fact that  
$ Range(P_{\epsilon, \delta})$ has finite dimension, it follows that, for a fixed $\delta$, 
 \begin{equation}\label{proy} \limsup_{n \to \infty} 
\| M_{n,\delta}^* \circ P_{\epsilon, \delta} \|
_{\H(\partial \Omega) \to H^{-1/2}(\partial \Omega)} \le C\delta^{-1} \limsup_{n \to \infty}  
 \| \sigma_n-\sigma\|_{L^{\infty}(\Omega_{2\delta})}.
\end{equation}
Note that the constant $C$ in the above estimate is chosen independent of $\e, \, \delta$.

    Thus,  
from \eqref{1-proy} and \eqref{proy}, we have
\[ \limsup_{n \to \infty} \| M_{n,\delta}^* \|_{\H(\partial \Omega) \to H^{-1/2}(\partial \Omega)} 
\le C\delta^{-1} \limsup_{n \to \infty}  
 \| \sigma_n-\sigma\|_{L^{\infty}(\Omega_{2\delta})}  + \epsilon, \]
for arbitrary $\epsilon>0$. Since 
\[
\| M_{n,\delta}^* \|_{\H(\partial \Omega) 
\to H^{-1/2}(\partial \Omega)}=
\| M_{n,\delta} \|_{\H(\partial \Omega) \to H^{-1/2}(\partial \Omega)},
\]
 the claim follows.  
 \end{proof}

To complete the proof of Theorem \ref{general}  ,
by Lemmata \ref{tres}, \ref{DV} and \ref{Poperator} we get that, for every $\delta>0$,
\[\limsup_{n \to \infty} \|\Lambda_{\sigma_n}-\Lambda_{\sigma} \|_{\H(\partial \Omega) \to H^{-1/2}(\partial \Omega)} 
\le 
C \delta^{-1} \limsup_{n \to \infty}  
 \| \sigma_n-\sigma\|_{L^{\infty}(\Omega_{2\delta})} . \]

We complete this section by an example which shows that,
in order to achieve the uniform convergence of the DN maps, one 
should indeed control the behavior of $\sigma_n$ in 
some vicinity of $\p \Omega$ and to see, in fact,   that the control of $\sigma_n$ and even all its derivatives on $\p \Omega$ is not sufficient.

\begin{Theorem} For any $\alpha >0$,
there exists a sequence $\sigma_n\in M_{1+\alpha}(B(0,1))$,   $\sigma_n=I$ on $\Omega_{\delta_n},\,\delta_n \to 0 $, such that $\sigma_n \to I$ in the sense of the
$G-$convergence,  but 
 \[ \limsup \|(\Lambda_{\sigma_n}- \Lambda_{I})\|_{\H \to H^{-1/2}}
>\frac{\alpha }{16(2+\alpha)}. \]
\end{Theorem}

\begin{proof}
Take  $ \alpha> 0$, $\Omega=B(0,1)\subset \mathbb R^2$, and   
 consider  the family of isotropic  conductivities,
 $\sigma_R,\, R<1,$   
$$\sigma_R= \gamma_R \,I=(\chi_{B(0,1)}(x)+ \alpha\chi_{B(0,R)-B(0,R^2)}(x))\,I$$
For these conductivities   we  have, on one hand, that
when $R \to 1$, then $\sigma_R$  $G$-converge  to $I$. On the  other hand,  we have  the  expression,
$$\langle(\Lambda_R- \Lambda_1)e^{ik\theta}, e^{il\theta}\rangle =
\delta_k^l |k|m_k,$$
where 
$$m_k=\frac{2\alpha(2+\alpha)(R^{2|k|}-R^{4|k|})}{(2+\alpha)^2- \alpha^2R^{2|k|}-\alpha(2+\alpha)(R^{2|k|}-R^{4|k|})}.$$
Taking the $H^\alpha-$norm on $\p B(0,1)$ of the form
\[
\|u\|^2_{H^\alpha}= |u_0|^2+
\sum_{k\neq 0}|k|^{2\alpha} |u_k|^2,\quad \hbox{if}\,\,
u= \sum u_k e^{ik\theta},
\]
we thus have
$$\|(\Lambda_{R}- \Lambda_1)e^{ik\theta}\|_{H^{-1/2}}=|k|^{1/2} m_k= |m_k| \|e^{ik\theta}\|_{H^{1/2}}.$$
Then, 
assuming $R>(3/4)^{1/4} $  and choosing
$k=[\frac{-1}{2\log_2 R}] $, we see that 
$$
\|(\Lambda_{R}- \Lambda_1)\|_{\H \to H^{-1/2}}
>\frac{\alpha }{16(2+\alpha)}.
$$
Hence, choosing $\sigma_n=\sigma_{R(n)},\, R(n) \to 1$ as $n \to \infty$, we see that
there is no convergence $\Lambda_n\to \Lambda_I$ in the operator  norm. 
\end{proof}

\end{section}

\begin{section} {Stability with respect to the G-convergence}
Stability with respect to the  G-convergence has been proved in 2D 
by Alessandrini and Cabib  in \cite{AC2} assuming further that 
$\nabla \cdot \sigma=0$. As discussed also in that paper, the lack of 
uniqueness in the Calder\'on problem in the anisotropic case prevents 
stability in the general case. Compactness arguments show that this is 
the only obstruction.   Moreover in dimension 2, it is known \cite[Theorem 1]{ALP} 
that the lack of uniqueness is due to a
 quasiconformal change of variables.  Since changes of variables preserve the G-convergence, the unconditional stability, 
 in the introduction, follows.  We will first recall the basic definitions, then prove that indeed  the 
G-convergence is preserved by the change of variables
  and finally will combine it all to prove Theorem~\ref{stability}.

For a constant $K \ge 1$, 
a $K$-quasiconformal mappings is a homeomophism $F:\mathbb{C} \to \mathbb{C}$ which belongs to 
$W^{1,2}_{\textrm{loc}} (\C)$ and such that 
\begin{equation}\label{ChangeXX}
\|DF(x)\|^2 \le K J_F,\quad J_F= \det(DF).
\end{equation}

Given $\sigma \in M_K(\Omega)$,
 its associated quadratic form $ l_\sigma:H_0^{1}(\Omega)   \to \mathbb{R}$ is
 defined by 
 \[ l_\sigma[u]=\int_\Omega \langle  \sigma \nabla u, \nabla u \rangle, \]
cf. (\ref{q-form}). 
Let $F=I$ at $\partial \Omega$. Then
$F_*(\sigma)$ is formally given by 

\begin{equation} l_{F_*(\sigma)} ( u,v)=l_{\sigma}  (F^*(u),F^*(v))= l_{\sigma}( u \circ F, v\circ F). \end{equation}
Expressing the push forward $F_*$ in coordinates, we have

\[F_*(\sigma)(y)= J_F^{-1}(x) DF(x) \sigma  DF(x)^t_{|F^{-1}(y)=x}. \]
It is straightforward to see that that, if $F$ is $K-$quasiconformal, then

\begin{equation}\label{K2K2}
F_*:M_{K}(\Omega) \to M_{K^2}(\Omega).
\end{equation}
Together with the fact that 
$F=I$ on $\partial \Omega$, this implies that $F_*$ is bijective in $H^1_0(\Omega)$
  and, by duality, in $H^{-1}(\Omega)$. 
Explicitly, if $f \in H^{-1}(\Omega)$ 
we solve 

\begin{equation}\label{ChangeXXI}
\triangle v=f,\quad f \in H^{-1}(\Omega) \quad \hbox{ with}\,\,v \in H_0^{1}(\Omega),
\end{equation}
and notice that
\begin{equation}\label{Fstar}
F^*(f)=  \nabla \cdot (F_*(I) \nabla (v \circ F^{-1} )) \in H^{-1}(\Omega).
\end{equation}
Now, since $\langle f, F^*(\varphi) \rangle= \langle F^*(f), \varphi \rangle$ 
it follows that  

\begin{equation} \label{commutes} L_{F_*(\sigma)}^{-1}(F^*(f)) = F^*(L_\sigma^{-1}(f))   \end{equation}

\begin{Lemma}\label{preserved}
Let $F$ be a quasiconformal homeomorphism fixing the boundary of a planar domain $\Omega$. Let
 $\sigma_n, \sigma \in M_{K}(\Omega)$. Then $\sigma_n \to \sigma$ in the sense of G convergence if and only 
 if $F_*(\sigma_n) \to F_*(\sigma)$ in the sense of G convergence. 
\end{Lemma}
\begin{proof}
By the definition of the $G$-convergence and (\ref{commutes}) it is enough to show that if $u_n \in H_0^1(\Omega)$ converges
weakly to $u$ then $F^*(u_n)$ converges weakly to $F^*(u)$.   Since 
 $F_*: H^1_0(\Omega) \to H^1_0(\Omega)$,
 there is $C>0$ such that
\[
\|u_n \circ F^{-1}\|_{H^1_0} \leq C.
\]

By the weak compactness of  $H^1_0(\Omega)$, it follows that  the sequence
$ u_n \circ F^{-1}$ has a subsequence converging weakly in $H^1_0$ and, therefore, 
strongly
in $L^2$, to some $ v \in H^1_0$.  On the other hand, the subsequence of $u_n$ has a further 
subsequence which converges almost 
everywhere to $u$ and thus $F^*(u_n) \to F^*(u)$ almost 
everywhere. 
By the uniqueness of the weak limits $v=F^*(u)$. 
\end{proof}

\emph{Proof of Theorem~\ref{stability}}
Due to Lemma \ref{Gconvergencia}, it is sufficient to prove that (\ref{29.1.07})
implies (\ref{3.17.4}).

Recall that $M_K$ is compact (\cite[Theorem 22.3]{D}) and metrizable (\cite[Corollary 10.23]{D})
with respect to the G-topology. Hence, if (\ref{29.1.07}) is valid, the sequence  
$\sigma_n$ (and any its subsequence)  
has a subsequence $\sigma_{n(k)},\, k=1, 2, \dots,$ 
which converges in 
the $G-$sense to a limit conductivity denoted 
by $\tilde{\sigma},\,
\tilde{\sigma} \in M_K$. It
then follows from the weak definition of the DN map,
(\ref{weak formulation}), that
\[
\Lambda_{\sigma}=\Lambda_{\tilde{\sigma}}.
\] 
Thus, since the $G$-topology is metrizable, if we define, for $\tilde K>1$,
\begin{equation} \label{2.17.4}
M_{\tilde K}(\sigma)=\{ \tilde{\sigma} \in 
M_{\tilde K}:\, \Lambda_{\sigma}=\Lambda_{\tilde{\sigma}} \}, 
\end{equation}
it follows that 
\[
\lim_{n \to \infty} d_G(\sigma_n, M_{K}(\sigma))=0,
 \]
where $d_G$ is a distance in $M_{K}$ inducing the $G-$topology. 

The above arguments work in any dimension but 
in 2D the sets $M_{\tilde K}(\sigma)$ are described in \cite[Theorem 1]{ALP}. Namely, 
\begin{equation} \label{2.17.5}
M_{\tilde K}(\sigma)=\{ \tilde{\sigma} \in M_{\tilde K}:\,  
\tilde{\sigma} =F^*(\sigma)\text{ for some quasiconformal map }F \text{ with } F_{|\partial \Omega}=I\}, 
\end{equation}

Thus, it is enough to prove 
that
$$\lim _{n\to \infty}d_G(\sigma, M_{K^2}(\sigma_n))=0,$$
see (\ref{K2K2}).
We prove this by contradiction. Assume that
  there is $\e >0$ and a subsequence $n(k)$, such that
\[
d_G(\sigma, M_k) >\e,\text{ where } \quad M_k= M_{K^2}(\sigma_{n(k)}).
\]
Since $M_{K^2}(\Omega)$ is G-compact
there is a non-relabelled subsequence of $\sigma_{n(k)}$ which is $G-$convergent
to $\tilde \sigma \in M_{K^2}$.  Since, for any $\tilde \sigma_n \in M_{K^2}(\sigma_n)$,
we have $\Lambda_{\tilde \sigma_n}=\Lambda_{ \sigma_n}$, and $\Lambda_{ \sigma_n} \to 
\Lambda_{ \sigma}$ in the weak sense, we see that $\Lambda_{\tilde \sigma}=\Lambda_{ \sigma}$.
By \cite{ALP}, there is a quasiconformal $F, \,$
with $F=\hbox{id}$
on $\p\Omega$, such that
\[
F_*(\tilde \sigma)= \sigma.
\]

Now Lemma~\ref{preserved} implies that $F_*(\sigma_{n_k}) \in M_k $ converges to $\sigma$.
This is a contradiction. 

\end{section}

\end{document}